\documentclass[]{article}
\title{The Pompeiu problem for isoparametric foliations}
\title{Isoparametric foliations and the Pompeiu problem}
\author{Luigi Provenzano and Alessandro Savo}

\date{}
\usepackage{makeidx}
\usepackage{setspace}
\usepackage{amssymb , amsmath, amsthm,graphicx, float,enumerate}
\usepackage{xcolor}
\newtheorem{defi}{Definition}
\newtheorem{thm}[defi]{Theorem}
\newtheorem{theorem}[defi]{Theorem}
\newtheorem{example}[defi]{Example}

\newtheorem{rem}[defi]{Remark}
 \newtheorem{prop}[defi]{Proposition}

\newtheorem{lemma}[defi]{Lemma}
\newtheorem{cor}[defi]{Corollary}

\newcommand{\twosystem}[2]{\left\{\begin{aligned} &#1\\ &#2\end{aligned}\right.}

\newcommand{\matrice}{\begin{pmatrix}}
\newcommand{\ok}{\end{pmatrix}}

\newcommand{\scal}[2]{\langle{#1},{#2}\rangle}

\newcommand{\abs}[1]{\lvert{#1}\rvert}

\newcommand{\reals}{{\mathbb R}}
\newcommand{\sphere}[1]{{\mathbb S}^{#1}}
\newcommand{\real}[1]{{\mathbb R}^{#1}}

\newcommand{\bd}{\partial}

\newcommand{\derive}[2]{\dfrac{\bd #1}{\bd#2}}

\begin{document}
\maketitle
\large
\makebox[\textwidth][c]{
\begin{minipage}{14cm}
{\centerline{\small\bf Abstract} }
\smallskip
\begin{spacing}{0.8}
{\small A bounded domain $\Omega$ in a Riemannian manifold $M$ is said to have the Pompeiu property if the only continuous function which integrates to zero on $\Omega$ and on all its congruent images is the zero function. In some respects, the Pompeiu property can be viewed as an overdetermined problem, given its relation with the Schiffer problem. It is well-known that every Euclidean ball fails the Pompeiu property while spherical balls have the property for almost all radii (Ungar's Freak theorem).  In the present paper we discuss the Pompeiu property when $M$ is compact and admits an isoparametric foliation. In particular, we identify precise conditions on the spectrum of the Laplacian on $M$ under which the level domains of an isoparametric function fail the Pompeiu property. Specific calculations are carried out when the ambient manifold is the round sphere, and some consequences are derived. Moreover, a detailed discussion of  Ungar's Freak theorem and its generalizations is also carried out.}

\medskip


\end{spacing}
\end{minipage}
}

\section{Introduction}

\subsection{Historical facts} Let $(M,g)$ be a Riemannian manifold with isometry group $G$. 
A bounded domain $\Omega$ of $M$ is said to have the {\it Pompeiu property} if the only continuous function $f$ on $\Omega$ such that
$$
\int_{h(\Omega)}f=0
$$
for all $h\in G$ is the identically zero function $f=0$. Here the integral is taken with respect  to the induced Riemannian measure. Equivalently, consider the linear map
$$
\mu: C^0(M)\to C^0(G), \quad f\mapsto\mu_f
$$
where 
$
\mu_f(h)\doteq \int_{h(\Omega)}f
$. 
Then, $\Omega$ has the Pompeiu property if and only if $\mu$ is injective.

\smallskip

Pompeiu thought that every Euclidean domain has the property that later would bring his name, and actually provided a (wrong) proof of this fact in 1929, see \cite{pompeiu}. Some years later, in 1944, Chakalov \cite{chakalov} showed that every Euclidean ball fails the Pompeiu property. Other examples of domains failing the property are obtained by removing balls of suitable radii from a larger ball, but no other examples were found so far. Hence the following conjecture is still open, to the best of our knowledge:

\smallskip

{\bf Pompeiu conjecture.} {\it Let $\Omega$ be a Euclidean domain with connected boundary. If $\Omega$  fails  the Pompeiu property, then it is a ball}.

\smallskip

The Pompeiu property can be viewed as an overdetermined problem; in fact, at least in $\real n$, the Pompeiu conjecture is equivalent to the so-called: 

\smallskip
{\bf Schiffer conjecture.}
{\it Let $\Omega\subseteq\real n$ be a domain with connected boundary which supports a non-trivial solution to the overdetermined PDE, called {\rm Schiffer problem}:
$$
\twosystem
{\Delta u=\lambda u\quad\text{on $\Omega$}}
{\derive uN=0, \quad u=c\quad\text{on $\bd\Omega$}}
$$
for some $\lambda>0$. Then $\Omega$ is ball.}

\smallskip

Here $\Delta$ is the positive Laplacian (in $\real n$: $\Delta u=-\sum_j\partial^2_{jj}u$).

The equivalence was proved by Williams \cite{williams1,williams2}, see also Berenstein \cite{berenstein0}. Schiffer conjecture is still open, although it was proved under various additional conditions; for a thorough exposition we refer to the survey papers \cite{dalmasso_pompeiu,zalcman,zalcman_bib}. 

\smallskip

The scope of this paper is to explore the Pompeiu problem on other Riemannian manifolds; in particular, on closed manifolds supporting isoparametric foliations, see Subsection \ref{sub:isof} for the relevant definitions. We will then focus on the round sphere, where a complete classification of isoparametric foliations is now available, see Section \ref{sec:sphere}. The reason why we study the Pompeiu property in terms of isoparametric foliations is because of their role in overdetermined PDE's, recently clarified in the papers \cite{savo_heat_flow,savo_geometric_rigidity,savo_heat_cont,shklover}. 

\smallskip

Perhaps the first significant result in a manifold different from $\real n$ was obtained by Ungar \cite{ungar} in 1954, who proved the following fact. Here $B(r)$ denotes the geodesic ball in $\sphere 2$ centered at any chosen point $x_0$ and having radius $r$:

\smallskip

{\bf Freak theorem.} {\it The set $S\subseteq (0,\pi)$ of radii $r$ for which $B(r)$ fails the Pompeiu property is countable and dense in $(0,\pi)$. 
In particular, for any $r$ in the complement of $S$, the ball $B(r)$ has the Pompeiu property.} \smallskip 

The word {\it freak} suggests that the result is quite surprising; in fact a consequence of the theorem is that a geodesic ball $B(r)$  in $\sphere 2$ has the Pompeiu property with probability $1$ for $r\in (0,\pi)$. Note the striking difference with the Euclidean case, where every ball fails the Pompeiu property, so that the set of radii $r$ for which a ball has the Pompeiu property is actually empty. 
The difference could be perhaps justified by the fact that $\sphere 2$ is compact, see Section \ref{A} for further discussions in dimension $1$. Later the result was extended to $\sphere n$ and more generally to compact symmetric spaces of rank one (see \cite{berenstein_zalcman_1,berenstein_zalcman_2}). We will reprove this in Section \ref{sec:freak}. 

\smallskip

It should be said that Ungar constructed other spherical domains which do not have the Pompeiu property: these are suitable polygonal regions with an appropriate number of edges. Thus, the Pompeiu conjecture fails on $\sphere 2$
(in fact, there are many spherical domains with smooth boundary failing the property, see Section \ref{sec:even}).

\smallskip

In the rest of the introduction we will briefly introduce the isoparametric foliations and other essential terminology, and state the main results. More details on the preliminaries will be given in Section \ref{sec:pre}.

\subsection{Isoparametric foliations} 

In this paper, we focus on compact Riemannian manifolds $M$ endowed with an isoparametric function, i.e., a smooth function $F:M\to [a,b]$ such that:
\begin{equation}\label{iso}
\begin{cases}
{\abs{\nabla F}^2=A\circ F}\\
{\Delta F=B\circ F}
\end{cases}
\end{equation}
for smooth functions $A,B:[a,b]\to \reals$. We call the foliation $\mathcal F$ of $M$ given by the level sets of $F$, namely
\begin{equation}\label{fol}
\mathcal F:\ M=\bigcup_{t\in[a,b]}F^{-1}(t),
\end{equation}
an {\it isoparametric foliation} of $M$. We will consider the pair $(M,\mathcal F)$. The regular level sets of $F$, that is, the sets $F^{-1}(t)$ for $t\in (a,b)$, are called {\it isoparametric hypersurfaces} of $M$. The first condition in \eqref{iso} insures that  these hypersurfaces are all parallel to one another, and the second says that they all have constant mean curvature. 

\smallskip

The  sets $M^+:=F^{-1}(a), M^-:=F^{-1}(b)$ are smooth, closed submanifolds of codimension possibly higher than $1$. They are called the {\it focal sets} of $\mathcal F$, and are always minimal in $M$. They are the singular leaves of the foliation. 

\smallskip

For simplicity, we assume in fact that the $M^+$ and $M^-$ have codimension greater than one (i.e., the foliation is {\it proper}): this will ensure that all leaves of the foliation, including the singular leaves, are connected. However, the main statements of the paper hold without this assumption. 

\smallskip

An immediate example of isoparametric foliation on $\sphere n$ is that given by concentric geodesic spheres centered at a fixed point $x_0$: in this case the focal sets are simply the north and south pole $\{x_0\}$ and $\{-x_0\}$. This can be generalized to compact {\it harmonic manifolds}  which, more or less by definition, are foliated by (smooth) geodesic spheres with constant mean curvature, centered at any fixed point.

\smallskip

We restrict to compact manifolds because we will work with the spectrum of the Laplacian on $M$, and we want it to be discrete in order to apply our arguments. Of course, isoparametric foliations exist also on non-compact manifolds; it was proved by Cartan \cite{cartan1} and Segre \cite{segre} that the foliation of $\real n$ (or $\mathbb H^n$) by concentric spheres is the only isoparametric foliation with compact leaves (up to congruences). More details on isoparametric foliations will be given in Subsection \ref{sub:isof}.

\smallskip

On the contrary, the sphere supports many interesting isoparametric foliations which are not congruent to the standard foliation by concentric spheres. We recall here that a hypersurface $\Sigma$ of $\sphere n$ is isoparametric (i.e., it belongs to an isoparametric family) if and only if it has constant principal curvatures, meaning that the characteristic polynomial of the second fundamental form is the same at all points of $\Sigma$.

\smallskip

After many intermediate results, the classification of (proper) isoparametric foliations of the sphere was completed only recently \cite{chi}. It turns out that these are classified in terms of the number $g$ of distinct principal curvatures of any of its leaves, which can only be $1,2,3,4,6$. Further details will be given in Section \ref{sec:sphere}. 

\smallskip

To better study the geometry of the foliation, we can re-normalize $F$ and consider in its place the distance function to $M^+$, which we denote by
\begin{equation}\label{distance}
\rho:M\to [0,D], \quad \rho(x):={\rm dist}(x, M^+),
\end{equation}
where $D={\rm dist}(M^+,M^-)$. We will call $D=D(\mathcal F)$ the {\it diameter} of the foliation: it is the maximal distance between two leaves.  When the foliation is the standard one by concentric geodesic spheres, $D(\mathcal F)$ is in fact the diameter of the manifold. 

\smallskip

The function $\rho$ is smooth on $M\setminus\{M^+\cup M^-\}$. Thus, the leaves of the foliations are also given by the level sets of $\rho$, see \eqref{fol2}. We refer to Subsection \ref{sub:dist} for more details on the distance function from a focal set.


\subsection{Isoparametric tubes and radial spectrum}  Given an isoparametric foliation $\mathcal F$ and $t\in (0,D(\mathcal F))$, the level domain
$$
\Omega_t=\{x\in M:\rho(x)<t\}
$$
is called an {\it isoparametric tube} of $\mathcal F$: it is the set of points at distance less than $t$ to the focal set $M^+$. Note that when the focal set is a point (i.e., the leaves of the foliations are geodesic spheres), an isoparametric tube is simply a geodesic ball. Needless to say, one could choose the focal set $M^-$ instead: therefore, the interior of the complement of an isoparametric tube is itself an isoparametric tube as well. 

\smallskip

In \cite{savo_heat_cont} we defined {\it isoparametric tube} any Riemannian domain which is a smooth, solid tube around a smooth, compact submanifold $\Sigma$, with the additional property that any equidistant from $\Sigma$ is a smooth hypersurface with constant mean curvature. This definition is slightly more general than the one adopted in this paper, because it allows $\Sigma$ to have codimension $1$ (at any rate, $\Sigma$ has to be minimal).  For a complete classification in the round sphere, see \cite{savo_heat_cont}. This class of Riemannian domains turns out to coincide with the class of domains supporting  solutions to  certain overdetermined problems involving the heat equation (in the general Riemannian case); in particular, it coincides with the class of domain which are critical for the heat content functional (see \cite[Theorem 4]{savo_heat_cont}). In this paper we have chosen to restrict a bit the definition for the sake of concreteness; however the main results carry over without change.

\smallskip

The main point we want to make is that very often, under certain conditions, isoparametric tubes on a compact manifold  {\it fail}  the Pompeiu property, pretty much as balls  in the Euclidean case. 

To explain such conditions, which will involve the spectrum of the Laplacian on $M$, we single out the vector space of {\it radial functions}, that is, the functions which depend only on the distance to  the focal set $M^+$ of $\mathcal F$. Obviously this notion depends on the foliation  $\mathcal F$ chosen; any radial function can be expressed as
\begin{equation}\label{radial}
f=\psi\circ\rho,
\end{equation}
for some function $\psi:[0,D(\mathcal F)]\to\reals$. Let 
$$
0=\lambda_1<\lambda_2\leq \cdots\leq\lambda_k\leq\cdots\nearrow+\infty.
$$
be the spectrum of the Laplacian on $(M,g)$, each eigenvalue being repeated according to its finite multiplicity. We shall denote by ${\rm Spec}_1(M)$ the set of all the eigenvalues of the Laplacian without multiplicity, i.e., the set of all distinct (real) values assumed by the eigenvalues of the Laplacian. 
\begin{defi}[Radial eigenvalue]\label{radialeig}
Given an isoparametric foliation $\mathcal F$ on $M$, we say that an eigenvalue $\lambda$ is {\rm radial} if there exists a radial eigenfunction associated to $\lambda$. 
\end{defi}
Radial eigenfunctions are characterized by an ODE on the function $\psi$ in \eqref{radial}, see equation \eqref{ODE_f}; the radial eigenvalues form an infinite subset of ${\rm Spec}_1(M)$, denoted by ${\rm Spec}(M,\mathcal F)$, and are the eigenvalues of a one-dimensional Sturm-Liouville problem associated to $\mathcal F$ and defined on the interval $(0,D(\mathcal F))$.


\subsection{Main theorems, general case}

We are now in the position to state our main theorems. The following two results show that, associated to any isoparametric foliation $\mathcal F$ of $M$, there is always a countable dense subset $S(\mathcal F)\subseteq (0,D(\mathcal F))$ such the family of isoparametric tubes 
$$
\{\Omega_t: t\in S(\mathcal F)\}
$$
fails the Pompeiu property. Actually, one could replace $S(\mathcal F)$ by the whole interval $(0,D(\mathcal F))$ provided that the radial spectrum is a proper subset of ${\rm Spec}_1(M,g)$.

\smallskip

More in detail, we have:

\begin{thm}\label{firstmain} Let $(M,\mathcal F$) be an isoparametric foliation. Assume that ${\rm Spec}(M,\mathcal F)$ is a proper subset of ${\rm Spec}_1(M)$. Then every isoparametric tube of $\mathcal F$ fails the Pompeiu property. 
\end{thm}

More precisely, let $f$ be an eigenfunction of $M$ associated to $\lambda\in {\rm Spec}_1(M)\setminus{\rm Spec}(M,\mathcal F)$. For all $t\in (0,D(\mathcal F))$, let $\Omega_t$ be the isoparametric tube bounded by $\rho^{-1}(t)$. Then,  for any isometry $h$ of $M$ one has:
$$
\int_{h(\Omega_t)}f=0.
$$

\smallskip

The second possibility is that ${\rm Spec}(M,\mathcal F)={\rm Spec}_1(M)$. Then also in this case we still have infinitely many isoparametric tubes failing the Pompeiu property. Precisely, consider the set
$S(\mathcal F)$, union of all zeroes of Sturm-Liouville eigenfunctions associated to $\mathcal F$. The set $S(\mathcal F)$ is countable and dense in $(0,D(\mathcal F))$, see \eqref{ct} for the precise definition. Then we have: 

\begin{theorem}\label{sf} Let $(M,\mathcal F$) be an isoparametric foliation.  Assume that ${\rm Spec}(M,\mathcal F)={\rm Spec}_1(M)$. Then, the isoparametric tube $\Omega_t$ fails the Pompeiu property for all $t\in 
S(\mathcal F)$.  
\end{theorem} 

\subsection{Freak theorem for two-point compact homogeneous spaces: a simpler proof}

One could ask if the converse of Theorem \ref{sf} holds, that is, is it true that if $t\notin S(\mathcal F)$ then $\Omega_t$ has the Pompeiu property? We don't have a proof of the converse, in general. We only remark here that, if $M$ is a compact two-point homogeneous space (or, equivalently, by classical results, a compact rank one symmetric space) and $\mathcal F$ is the standard foliation by geodesic spheres, then the converse is also true, see \cite{berenstein_zalcman_2}. We will actually give a simple proof of this fact based on a general Addition Formula (see Section \ref{sec:freak}). 

\smallskip

Here is the relevant statement:
\begin{theorem}\label{chs} Let $M$ be a compact two-point homogeneous space with diameter $D$.
There is a countable dense set $S\subset (0,D)$ such that the geodesic ball of radius $t$ fails the Pompeiu property  if $t\in S$ and has the Pompeiu property if  $t\in (0,D)\setminus S$.
\end{theorem}

In fact, the set $S$ is actually $S(\mathcal F)$ where $\mathcal F$ is the foliation with focal set $M^+$ being a point.

\subsection{Main theorems, the round sphere} 

We discuss now the above results for the most relevant case: the family of isoparametric foliations on the round sphere $\sphere n$. Isoparametric foliations on spheres have been most studied in the literature, and are divided in five classes, according to the number $g=g(\mathcal F)$ of distinct principal curvatures of any of its leaves.  The following is a classical result of Munzner \cite{munzner1,munzner2}.

\begin{theorem} Let $(\sphere n,\mathcal F)$ be an isoparametric foliation with $g$ distinct principal curvatures. Then:
\begin{enumerate}[i)]
\item The only possible values of $g$ are $1,2,3,4,6$.

\item The diameter of $\mathcal F$  is $D(\mathcal F)=\dfrac{\pi}{g}$.
\end{enumerate}
\end{theorem}

It is possible to compute explicitly the radial spectrum for any of such classes. We prove the following:

\begin{theorem}\label{thmsphere} Let $(\sphere n,\mathcal F)$ be an isoparametric foliation with $g$ distinct principal curvatures. Then its radial spectrum is
$$
{\rm Spec}(\sphere n,\mathcal F)=\{gk(gk+n-1):k\in\mathbb N\}.
$$
\end{theorem}

It is well-known that ${\rm Spec}_1(\sphere n)$ is given by the collection $k(k+n-1)$ for $k\in\mathbb N$. One then  sees that ${\rm Spec}(M,\mathcal F)={\rm Spec}_1(M)$ if and only if $g=1$, that is, if and only if $\mathcal F$ is the foliation by concentric geodesic spheres. Therefore, as a consequence of Theorem \ref{firstmain}, we obtain:
\begin{cor}\label{corsphere}
Let $\mathcal F$ be an isoparametric foliation of $\mathbb S^n$ with $g\ne 1$. Then every isoparametric tube of $\mathcal F$ fails the Pompeiu property.
\end{cor}

The case $g=1$ reduces to (half of) Ungar's Freak Theorem (in dimension $n$). 
Shklover \cite{shklover} proves a particular case of Corollary \ref{corsphere}: by using special functions, and the relation of the Pompeiu problem with Schiffer overdetermined problem, he proves that the set of isoparametric tubes $\Omega_t$ failing the Pompeiu property is countable and dense in the interval $t\in (0,\frac{\pi}{g})$. Corollary \ref{corsphere} follows from Theorem \ref{firstmain}; on the sphere it improves Shklover's result, and is conceptually simpler.

Actually, when the ambient manifold is the sphere, we give a very simple argument showing that every antipodal invariant domain in $\sphere n$ fails the Pompeiu property 
(see  Proposition \ref{peven}): hence, as  isoparametric tubes with $g$ even are antipodal invariant, this argument reproves Corollary \ref{corsphere}  when $g=2,4,6$. 

\smallskip

A consequence of Corollary \ref{corsphere} is that the barycenter of every isoparametric hypersurface having $g\ne 1$ is always the origin. By standard min-max methods one then obtains the following upper bound on $\lambda_2(\Sigma)$, the first positive eigenvalue of the Laplacian on an isoparametric hypersurface $\Sigma$ (see Section \ref{yau}).

\begin{theorem}\label{thmyau}
Let $\Sigma$ be a connected isoparametric hypersurface of $\sphere n$ with $g>1$ distinct principal curvatures. Then 
$$
\lambda_2(\Sigma)\leq n-1.
$$
Equality holds if and only if $\Sigma$ is minimal. 
\end{theorem}

We  have equality in the theorem thanks to \cite{tang_yan_yauconj_2,tang_yan_yauconj_1}, which verifies Yau's conjecture  
for minimal isoparametric hypersurfaces.

In view of the above theorem, as a final remark, assume that $\Sigma_t=F^{-1}(t)$ belongs to the isoparametric foliation $\mathcal F$ defined by the Cartan polynomial $F$. Then,  it is reasonable to expect that $\lambda_2(\Sigma_t)$ is an increasing function for $t\in[0,t_{min}]$, where $\Sigma_{t_{min}}$ is the unique minimal representative of $\mathcal F$. We searched the literature for results like these, without success so far. 

\smallskip

The present paper is organized as follows. In Section \ref{sec:pre} we recall a few preliminary results on isoparametric functions and foliations, and introduce the radial spectrum. In Section \ref{sec:pompeiuprop} we prove Theorem \ref{firstmain} (see Subsection \ref{sub:sufficient}) and  Theorem \ref{sf} (see Subsection \ref{sub:density}). In Section \ref{sec:freak} prove Theorem \ref{chs}, namely Ungar's Freak Theorem for two-point compact homogeneous spaces. In Section \ref{sec:sphere} we discuss Theorems \ref{firstmain} and \ref{sf} in the case of the round sphere $\mathbb S^n$, and prove Theorem \ref{thmsphere} and Corollary \ref{corsphere}. In Section \ref{sec:even} we prove that the Pompeiu property fails on any antipodal invariant domain of the sphere. In Section \ref{yau} we prove Theorem \ref{thmyau}. Finally, in Section \ref{A} we collect a few elementary one-dimensional examples which allow an easier understanding of the results contained in the paper.

\section{Preliminary results}\label{sec:pre}
\subsection{Isoparametric functions on Riemannian manifolds}\label{sub:isof} 

We recall that an isoparametric function on a compact Riemannian manifold $(M,g)$ is a smooth function $F:M\rightarrow[a,b]$ such that $|\nabla F|^2=A\circ F$ and $\Delta F=B\circ F$ for some smooth functions $A,B:[a,b]\rightarrow \mathbb R$, see \eqref{iso}. The regular level sets of $F$ are called {\it isoparametric hypersurfaces}. For some historical background we refer to Cartan \cite{cartan1,cartan2}, Levi-Civita \cite{levi} and Segre \cite{segre}, among the first to have systematically studied this subject. 

\smallskip

We list here a few fundamental results on isoparametric functions, proved e.g., in \cite{wang_87}.

\begin{prop}\label{prop_1}
Let $(M,g)$ be a compact Riemannian manifold admitting an isoparametric function $F:M\rightarrow [a,b]$. Then:
\begin{enumerate}[i)]
\item the open interval $(a,b)$ consists of regular values of $F$;
\item the sets $M^+:=F^{-1}(a), M^-:=F^{-1}(b)$ are smooth, closed submanifolds called the {\rm focal sets} of $F$;
\item the regular level set $F^{-1}(t)$ are all  parallel and equidistant to both $M^+$ and $M^-$; they all have constant mean curvature.
\end{enumerate}
\end{prop} 

It turns out that $M^+,M^-$ are minimal submanifolds, see \cite{ge_tang_13,nomizu}. According to \cite{ge_tang_13} we have the following:

\begin{defi}[Proper isoparametric functions]
An isoparametric function $F:M\rightarrow[a,b]$ is said to be {\rm proper} if the focal sets have codimension at least $2$ in $M$.
\end{defi}
A proper isoparametric function satisfies the following additional properties.
\begin{prop}\label{prop_2}
Let $(M,g)$ be a compact Riemannian manifold admitting a proper isoparametric function $F:M\rightarrow [a,b]$. In addition to properties i)-iii) of Proposition \ref{prop_1} we have:
\begin{enumerate}[i)]
\item $M^+,M^-$ are connected as well as all the regular level sets;
\item at least one regular level set is a minimal hypersurface; if $M$ has positive Ricci curvature, it is unique.
\end{enumerate}
\end{prop} 

We refer to \cite{wang_87} for more information on isoparametric functions and hypersurfaces. We also refer to the survey paper \cite{thorbergsson_survey} for further historical information and for a quite complete collection of references.

\smallskip

{\bf Assumption.} From now on, we will assume that $M$ admits a proper isoparametric function $F:M\rightarrow [a,b]$.

\subsection{The distance function and isoparametric foliations}\label{sub:dist}

We can re-normalize $F$ and consider in its place the distance function to $M^+$, defined e.g., in \eqref{distance}. Namely, $\rho(x)={\rm dist}(x,M^+)$, where $M^+$ is one of the focal sets. The function $\rho$ is smooth on $M\setminus\{M^+\cup M^-\}$ and takes values in $[0,D(\mathcal F)]$, where $D(\mathcal F)={\rm dist}(M^+,M^-)$. 

\medskip
For $x\in M$, let $\Sigma_x$ denote the equidistant hypersurface to $M^+$ containing $x$, namely
$$
\Sigma_x=\{y\in M: \rho(y)=\rho(x)\}=\rho^{-1}(\rho(x)).
$$
Clearly $\Sigma_x$ is a smooth connected hypersurface of constant mean curvature if $x\in M\setminus\{M^+\cup M^-\}$. It is one of the focal varieties otherwise. In particular $\rho^{-1}(0)=M^+$, $\rho^{-1}(D(\mathcal F))=M^-$. 

\medskip

Let $x\in M\setminus\{M^+\cup M^-\}$. It is well-known that $\nabla\rho(x)$ defines a unit normal vector field to $\Sigma_x$. The function $\Delta\rho$, restricted to $\Sigma_x$, is constant and measures the mean curvature of $\Sigma_x$ (see also \eqref{theta}).

\medskip

Recall that the level sets of an isoparametric function $F$ generate an {\it isoparametric foliation} as described in \eqref{fol}. An isoparametric foliation $\mathcal F$ can be described equivalently using the level sets of $\rho$, namely,
\begin{equation}\label{fol2}
\mathcal F:\ M=\bigcup_{t\in[0,D(\mathcal F)]}\rho^{-1}(t).
\end{equation}
Accordingly, also $\rho^{-1}(t)$ will be called a {\it leaf} of the foliation $\mathcal F$.  From now on we will consider the pair $(M,\mathcal F)$ given by a compact Riemannian manifold $(M,g)$ (we shall omit the metric $g$) with an isoparametric foliation $\mathcal F$ on it.
 
 \medskip

\subsection{Isoparametric foliations and isometries}\label{sub:isom}
If $h$ is an isometry of $M$ and $F$ is an isoparametric function on $M$, it is readily seen that the function $h\cdot F\doteq F\circ h^{-1}$
is an isoparametric function as well, which generates a foliation denoted by $h\cdot \mathcal F$. We say that $h\cdot\mathcal F$ is {\it congruent} to $\mathcal F$. Clearly, the focal sets and the regular leaves of $\mathcal F$ are pairwise congruent to those of $h\cdot\mathcal F$, in the sense that $L$ is a leaf of $\mathcal F$ if and only if  $h(L)$ is a leaf of $h\cdot\mathcal F$.

\begin{rem}
Given a manifold $M$, it can admit many (possibly infinitely many) non-congruent isoparametric foliations. This is the case of $\sphere n$ as recalled in Section \ref{sec:sphere}.
\end{rem}

\subsection{The radial spectrum}\label{sub:radspec} Let $(M,\mathcal F)$ be an isoparametric foliation with focal varieties $M^+,M^-$. Given a function $f$ on $M$, we say that it is {\it radial} if it depends only on the distance $\rho$ to $M^+$, hence, if and only if $f=\psi\circ \rho$ for some function $\psi:[0,D(\mathcal F)]\to\reals$. 

\medskip

Let $f\in C^{\infty}(M)$. Averaging $f$ on the level sets of $\rho$ we obtain the {\it radialization of $f$}; it is the function denoted $\mathcal A f$ and defined as
$$
\mathcal Af(x):= \dfrac{1}{\abs{\Sigma_x}}\int_{\Sigma_x}f.
$$
Note that when $x$ belongs to the focal set, say $x\in M^+$, then $\Sigma_x=M^+$ and $\abs{\Sigma_x}$ and $\int_{\Sigma_x}f$ denote, respectively, the Riemannian measure of $M^+$ and the integral of $f$ on $M^+$, for the induced Riemannian measure.

\smallskip

As proved in \cite{savo_geometric_rigidity}, the radialization $\mathcal A f$ is a smooth function as well, and by definition it is radial. The definition of radialization provides an equivalent characterization of radial functions.
\begin{lemma}
A function $f$ is radial if and only if $f=\mathcal A f$.
\end{lemma}

The crucial property of an isoparametric foliation is given by the following theorem proved in \cite{savo_geometric_rigidity}.

\begin{thm} Let $(M,\mathcal F)$ be an isoparametric foliation. Then the radialization operator commutes with the Laplacian. That is, for all $f\in C^{\infty}(M)$ one has
$$
\Delta(\mathcal Af)=\mathcal A(\Delta f).
$$
In particular, the Laplacian preserves the subspace of radial functions.
\end{thm} 

Recall that the spectrum of the Laplacian on $M$ is given by
$$
0=\lambda_1<\lambda_2\leq \cdots\leq\lambda_k\leq\cdots\nearrow+\infty,
$$
where each eigenvalue is repeated according to its finite multiplicity. By ${\rm Spec}_1(M)$ we have denoted the set of all the eigenvalues of the Laplacian without multiplicity, i.e., the set of all distinct real values assumed by the eigenvalues of the Laplacian.

\smallskip

A radial eigenfunction associated with a radial eigenvalue as in Definition \ref{radialeig} satisfies an explicit ODE. In order to present such ODE, we first introduce the normal coordinates based on $M^+$, and 
we refer to \cite{savo_geometric_rigidity} for full details. As usual, $M^+$ is one of the two focal sets of the isoparametric foliation $\mathcal F$, and $\rho:M\to [0,D(\mathcal F)]$ is the distance function to $M^+$ which is smooth on 
$M\setminus \{M^+\cup M^-\}$.  Let $U(M^+)$ be the unit normal bundle of $M^+$; then $U(M^+)$ is locally isometric to $M^+\times \sphere{n-k-1}$, where $k=\dim M^+$, and we can write an element $\xi\in U(M^+)$ as a pair 
$$
\xi=(x,\nu(x)),
$$
where $x\in M^+$ and $\nu(x)$ is a unit vector in the tangent space of $M$ at $x$ normal to $M^+$ at $x$. We define the normal exponential map 
$$
\Phi: (0,D(\mathcal F))\times U(M^+)\to M\setminus \{M^+\cup M^-\}
$$
by $\Phi(t,\nu(x))=\exp_x(t\nu(x))$. The map $\Phi$ is a diffeomorphism; in fact, given any point $y\in M\setminus \{M^+\cup M^-\}$ let $\gamma_y$ be the geodesic which minimizes distance from $y$ to $M^+$: if $x\in M^+$ is the foot of such geodesic, and $t$ is the distance from $y$ to $M^+$, then one has $y=\Phi(x,\nu(x))$ where $\nu(x)=-\gamma_y(t)$. 

\smallskip

Let $dv_g$ be the Riemannian volume form of $M$; then, in normal coordinates, it writes
$$
\Phi^{\star}dv_g(t,\xi)=\theta(t,\xi) dtd\xi
$$
where $d\xi$ is the volume form of $U(M^+)$, for a smooth function $\theta$ defined on $ (0,D(\mathcal F))\times U(M^+)$, which is simply the density of the volume form in normal coordinates around $M^+$. Assume that the point 
$y\in M\setminus \{M^+\cup M^-\}$ has normal coordinates $(t,\xi)$. By \cite[Proposition 12]{savo_geometric_rigidity}  one has:
\begin{equation}\label{theta}
-\dfrac{\theta'(t,\xi)}{\theta(t,\xi)}=\Delta\rho(y)=H(y),
\end{equation}
where $\theta'$ denotes differentiation with respect to $t$ and $H(y)$ is the mean curvature of the equidistant $\Sigma_y$ containing $y$ (the mean curvature is intended to be the trace of the second fundamental form with respect to the unit normal vector $\nabla\rho$). As $\Sigma_y$ is an isoparametric hypersurface, it has by hypothesis constant mean curvature, so that the function on the right is constant on $\Sigma_y$, which is to say,  the expression on the left also does not depend on $\xi$. By integration, the density $\theta(t,\xi)$ is then independent on $\xi$ as well, and will be written $\theta(t)$, simply.

\smallskip

\begin{prop} An eigenfunction $f$ corresponding to a radial eigenvalue $\lambda$ is of the form $f=\psi\circ\rho$, where $\psi:(0,D(\mathcal F))\rightarrow\mathbb R$ is smooth and solves the ordinary differential equation
\begin{equation}\label{ODE_f}
\psi''+\frac{\theta'}{\theta}\psi'+\lambda\psi=0.
\end{equation}
on $(0,D(\mathcal F))$. Moreover,  $\psi'(0)=\psi'(D(\mathcal F))=0$.
\end{prop}

\begin{proof} We follow \cite[\S 2.1]{savo_geometric_rigidity}. We first see that $\Delta(\psi\circ\rho)=(\psi'\circ \rho)\Delta\rho-(\psi''\circ\rho)|\nabla\rho|^2$, and $|\nabla \rho|=1$ when $\rho\in(0,D(\mathcal F))$. Then, by \eqref{theta}, we have $\Delta\rho=-\frac{\theta'}{\theta}\circ\rho$. Moreover $\theta(t)>0$ for $t\in(0,D(\mathcal F))$. This establishes \eqref{ODE_f}. Finally, note that $\theta(0)=\theta(D(\mathcal F))=0$. Then, since $f$ is smooth, necessarily $\psi'(0)=\psi'(D(\mathcal F))=0$.
\end{proof}

\begin{example}
For $\sphere n$ and the standard foliation $\mathcal F$ given by concentric spheres, we have $\theta(\rho(x))=\sin^{n-1}(\rho(x))$, and $\Delta\rho(x)=-(n-1)\cot(\rho(x))$.
\end{example}

By $V(\lambda)$ we denote the eigenspace corresponding to an eigenvalue $\lambda$. As the radialization operator $\mathcal A$ commutes with the Laplacian, we see that $\mathcal A$ maps  $V(\lambda)$ to itself, namely
$$
\mathcal A(V(\lambda))\subseteq V(\lambda).
$$
Thus we have the following equivalent characterization: an eigenvalue $\lambda$ is radial if and only if
$$
\mathcal A(V(\lambda))\ne\{0\}.
$$
Clearly the constant function $1$ is radial, hence $\lambda_1=0$ is radial. Let us list the radial eigenvalues and denote them as follows:
\begin{equation}\label{rad_eig}
0=\lambda_1(M,\mathcal F)<\lambda_2(M,\mathcal F)<\cdots<\lambda_k(M,\mathcal F)<\cdots
\end{equation}

The set of all the radial eigenvalues is called the {\rm radial spectrum} of $(M,\mathcal F)$ and is denoted by ${\rm Spec}(M,\mathcal F)$.

\smallskip

We note that the radial eigenvalues listed in \eqref{rad_eig} are all distinct. This is a consequence of the following fact.

\begin{thm}\label{radial_spec} Let $(M,\mathcal F)$ be an isoparametric foliation. Then:
\begin{enumerate}[i)]
\item  ${\rm Spec}(M,\mathcal F)$ is an infinite subset of  ${\rm Spec}_1(M)$, which could possibly coincide with ${\rm Spec}_1(M)$.
\item Any eigenvalue of the radial spectrum has exactly one radial eigenfunction (up to scalar multiplication). 
\end{enumerate}
\end{thm} 

\begin{proof} First observe that $L^2(M)=L^2_{rad}(M)\oplus L^2_{rad}(M)^{\perp}$, where $L^2_{rad}(M)$ is the infinite dimensional subspace of radial functions in $L^2(M)$ and the orthogonality is with respect to the scalar product of $L^2(M)$. Since the Laplacian of a radial function is radial (and the Laplacian preserves the decomposition) we deduce the existence of a Hilbert basis of $L^2_{rad}(M)$ of (radial) eigenfunctions. If the radial eigenvalues form a finite set, then $L^2_{rad}(M)$ would have finite dimension, which is not the case: this proves $i)$. To prove $ii)$, we see from \eqref{ODE_f} that two radial eigenfunctions $f_1=\psi_1\circ\rho$ and $f_2=\psi_2\circ\rho$ associated to the same eigenvalue $\lambda$ satisfy
$$
\theta(t)(\psi_1(t)'\psi_2(t)-\psi_2(t)'\psi_1(t))=C\,,\ \ \ t\in(0,D(\mathcal F))
$$
Since $\psi_1,\psi_2$ are smooth, taking the limit as $t\rightarrow 0^+$ or $t\rightarrow D(\mathcal F)^-$ we find that $C=0$, hence $\psi_1,\psi_2$ are linearly dependent.
\end{proof}

An alternative approach to the proof of Theorem \ref{radial_spec} is to study the ordinary differential equation \eqref{ODE_f}. The weight $\theta$ is smooth and positive on $(0,D(\mathcal F))$, $\theta(t)\sim t^{m^+}$ as $t\rightarrow 0^+$, $\theta(t)\sim (D(\mathcal F)-t)^{m^-}$ as $t\rightarrow D(\mathcal F)^-$, for some positive integers $m^+,m^-$ depending on $n$ and on the codimension of $M^+,M^-$, respectively. Thus \eqref{ODE_f} is a singular equation which admits a self-adjoint realization when we require boundedness of solutions at the endpoints. The corresponding spectrum is discrete, made of simple, non-negative eigenvalues diverging to $+\infty$.  Since a bounded solution $\psi$ of \eqref{ODE_f} necessarily satisfies  $\psi'(0)=\psi'(D(\mathcal F))=0$, then $f=\psi\circ\rho$ is a radial eigenfunction. We refer to \cite[\S 10]{zettl_sturm_liouville} and references therein for more details. 

\medskip

Note that ${\rm Spec}(M,\mathcal F)={\rm Spec}_1(M)$ if and only if all eigenspaces admit a radial eigenfunction. In the case of $\sphere n$ and the foliation $\mathcal F$ given by concentric spheres, then ${\rm Spec}(\sphere n,\mathcal F)={\rm Spec}_1(\sphere n)$. This is no longer true for other non-congruent foliations of $\sphere n$ as proved in Section \ref{sec:sphere}.

\begin{thm} The radial spectrum does not depend on the isometry. In other words, for any isometry $h$ of $M$ one has
$$
{\rm Spec}(M,\mathcal F)={\rm Spec}(M,h\cdot\mathcal F)
$$ 
\end{thm} 
This theorem is a consequence of the fact that the radial spectrum is exactly the spectrum of the singular problem \eqref{ODE_f}: this problem depends only on the density function $\theta$, which is the same for all congruent foliations. 

\section{The Pompeiu property}\label{sec:pompeiuprop}

In this section we will relate the radial spectrum with the Pompeiu property on certain domains associated with isoparametric foliations, and we will prove Theorems \ref{firstmain} and \ref{sf}.

\smallskip

Let $(M,\mathcal F)$ be an isoparametric foliation, and let $\rho:M\to[0,D(\mathcal F)]$ be the distance to the focal set $M^+$. For $t\in (0,D(\mathcal F))$ we call the level domain
$$
\Omega_t=\{x\in M:\rho(x)<t\}
$$
an {\it isoparametric tube} of $\mathcal F$. Being the isoparametric foliation proper, a corresponding isoparametric tube is connected.

\subsection{A sufficient condition for the failure of Pompeiu property}\label{sub:sufficient}

We prove the first main theorem of the paper, namely Theorem \ref{firstmain}, which establishes sufficient conditions under which Pompeiu property fails for all isoparametric tubes.

\begin{thm}\label{first_main} Let $(M,\mathcal F$) be an isoparametric foliation. Assume that ${\rm Spec}(M,\mathcal F)\subsetneq{\rm Spec}_1(M)$. Then every isoparametric tube of $\mathcal F$ fails the Pompeiu property. 
More precisely, let $f$ be an eigenfunction of $M$ associated to $\lambda\in {\rm Spec}_1(M)\setminus{\rm Spec}(M,\mathcal F)$. Then, for all $t\in (0,D(\mathcal F))$ and for any isometry $h$ of $M$ one has:
$$
\int_{h(\Omega_t)}f=0.
$$
\end{thm}

\begin{proof} We argue by contradiction, and assume that there exists $t\in (0,D(\mathcal F))$ and an isometry $h$ of $M$ such that
$$
\int_{h(\Omega_t)}f\ne 0
$$
Consider the foliation $h\cdot\mathcal F$ and let $\rho$ be the distance function to the focal set $h(M^+)$ of $\Omega:= h(\Omega_t)$. Then, by the coarea formula
$$
\int_0^t\left(\int_{\rho^{-1}(r)}f \right)dr=\int_{\Omega}f\ne 0,
$$
which implies that there exists $r_0\in (0,D(\mathcal F))$ such that $\int_{\rho^{-1}(r_0)}f\ne 0$. But then, the radialization of $f$ with respect to $h\cdot\mathcal F$ is non-zero, and is a radial eigenfunction associated to $\lambda$, which would imply that
$\lambda\in {\rm Spec}(M,h\cdot\mathcal F)= {\rm Spec}(M,\mathcal F)$. This contradicts the assumptions. Then the theorem holds. 
\end{proof}



\subsection{Density of isoparametric tubes failing the Pompeiu property}\label{sub:density}

It could happen that ${\rm Spec}_1(M)={\rm Spec}(M,\mathcal F)$. This is the case, for example, of the foliation by concentric spheres on $\sphere n$. However, there is always a dense subset of radii for which the isoparametric tube fails the Pompeiu property.  This is exactly our second main result, namely Theorem \ref{sf}, which we will prove in this subsection. We need a few preliminary results.

\smallskip

\begin{lemma} Let $f$ be a (not necessarily radial) eigenfunction of the Laplacian associated to the (positive) eigenvalue $\lambda$. Define the function $\Psi:[0,D(\mathcal F)]\to\reals$ by
$$
\Psi(t)=\int_{\Omega_t}f
$$
where $\Omega_t=\{x: \rho(x,M^+)<t\}$ is an isoparametric tube associated to the foliation $\mathcal F$. 
Then:
\begin{enumerate}[i)]
\item $\Psi$ is smooth, and satisfies the boundary value problem:
\begin{equation}\label{bvpsi}
\begin{cases}
{\Psi''-\dfrac{\theta'}{\theta}\Psi'+\lambda\Psi=0}\,, & {\rm in\ }(0,D(\mathcal F))\\
{\Psi(0)=\Psi(D(\mathcal F))=0}.
\end{cases}
\end{equation}

\item The spectrum of \eqref{bvpsi} coincides with the spectrum  of problem \eqref{ODE_f}, that is, the radial spectrum ${\rm Spec}(M,\mathcal F)$.
\end{enumerate}
\end{lemma} 

\begin{proof} $i)$ It is clear that $\Psi(0)=0$; as $\Omega_{D(\mathcal F)}=M\setminus M^-$ we see that $\Psi(D(\mathcal F))=\int_Mf=0$ because any non-constant eigenfunction has zero mean over $M$. Now 
$
\Psi'(t)=\int_{\rho=t}f
$
and an easy application of the Green formula gives (see \cite[\S 2.1]{savo_geometric_rigidity}):
$$
\Psi''(t)=\int_{\rho=t}(\scal{\nabla f}{\nabla\rho}-f\Delta\rho).
$$
As $\nabla\rho$ is the exterior unit normal to $\Omega_t$, and $\Delta\rho=-\frac{\theta'}{\theta}\circ\rho$, we get from the above:
$$
\begin{aligned}
\Psi''(t)&=-\lambda\int_{\Omega_t}f+\frac{\theta'}{\theta}\int_{\rho=t}f\\
&=-\lambda \Psi+\frac{\theta'}{\theta}\Psi',
\end{aligned}
$$
which gives the assertion. 

$ii)$ The two eigenvalue problems are unitarily equivalent via the map $\psi\mapsto L\psi$ where
$$
L\psi(t)=\int_0^t\theta(r)\psi(r)\,dr.
$$
One verifies that $\psi$ is a solution of \eqref{ODE_f}  if and only if $\Psi=L\psi$ is a solution of $\eqref{bvpsi}$.
\end{proof} 

Given $\lambda_k\in{\rm Spec}(M,\mathcal F)$, consider a radial eigenfunction $\phi_k$ associated to $\lambda_k$ and the content function 
\begin{equation}\label{psi}
\Psi_k(t)=\int_{\Omega_t}\phi_k, \quad\text{for all}\quad t\in [0,D(\mathcal F)],
\end{equation}
Then, $\Psi_k$ is a non trivial solution of problem \eqref{bvpsi}, hence an eigenfunction  associated to $\lambda_k$:
$$
\begin{cases}
{\Psi_k''-\dfrac{\theta'}{\theta}\Psi'_k+\lambda_k\Psi_k=0}\,, & {\rm in\ }(0,D(\mathcal F)),\\
{\Psi_k(0)=\Psi_k(D(\mathcal F))=0}.
\end{cases}
$$
We denote by $S_k$ the interior zero set of $\Psi_k$:
$$
S_k(\mathcal F)=\{t\in (0,D(\mathcal F)): \Psi_k(t)=0\}.
$$
Since every $\lambda_k$ is simple, $S_k(\mathcal F)$ does not depend on $\phi_k$.
By Courant-type arguments, the cardinality of $S_k(\mathcal F)$ cannot exceed $k$.  We then define:
\begin{equation}\label{ct}
S(\mathcal F)=\cup_{k=1}^{\infty}S_k(\mathcal F).
\end{equation}

It is clear that $S(\mathcal F)$ is countable and does not depend on the isometry, that is 
$S(\mathcal F)=S(h\cdot\mathcal F)$ for any isometry $h$.

We are now in position to prove our second main theorem:

\begin{theorem}\label{s_f}
Let $(M,\mathcal F$) be an isoparametric foliation and let $S(\mathcal F)$ be the subset of $(0,D(\mathcal F))$ defined by \eqref{ct}. Then
\begin{enumerate}[i)]
\item The isoparametric tube $\Omega_t$ fails the Pompeiu property for all $t\in 
S(\mathcal F)$.  
\item The set $S(\mathcal F)$ is countable and dense in $(0,D(\mathcal F))$. 
\end{enumerate}
\end{theorem} 
\begin{proof} $i)$ If $t\in S_k(\mathcal F)$ then $\Psi_k(t)=0$ that is,
$
\int_{\Omega_t}\phi_k=0.
$
To show that $\Omega_t$ fails the Pompeiu property, it is enough to show that, for any isometry $h$, one has:
$$
\int_{h(\Omega_t)}\phi_k=0.
$$
Fix an isometry $h$ and consider the congruent foliation $h\cdot\mathcal F$, with focal set $\tilde M^+=h(M^+)$. Let $\tilde\rho$ be the distance function to $\tilde M^+$ and consider the isoparametric tube $\tilde\Omega_t=\{\tilde\rho<t\}$. It is clear that
$\tilde\Omega_t=h(\Omega_t)$ and, since the mean curvature of the leaves $\{\rho=t\}$ and $\{\tilde\rho=t\}$ are the same, we see that the density functions $\theta(t)$ and $\tilde\theta(t)$ are equal. Therefore the content function
$$
\tilde\Psi_k(t)=\int_{\tilde\Omega_t}\phi_k=\int_{h(\Omega_t)}\phi_k
$$
is also a solution to problem \eqref{bvpsi}. Since the eigenspace of $\lambda_k$ is one dimensional, we conclude that there is $c(h)\in\reals$ such that
$$
\int_{h(\Omega_t)}\phi_k=c(h)\int_{\Omega_t}\phi_k. 
$$
Note that $c(h)$ could be zero, but at any rate we see that if $\int_{\Omega_t}\phi_k=0$ for some $t$  then $\int_{h(\Omega_t)}\phi_k$ for all $h$. This proves that $\Omega_t$ fails the Pompeiu property for all $t\in S(\mathcal F)$.

\smallskip
 
$ii)$ We already remarked that $S(\mathcal F)$ is countable. We prove that $S(\mathcal F)$ is dense in $(0,D(\mathcal F))$. Assume that the complement of $S(\mathcal F)$ contains an open interval $(a,b)$.

\smallskip

By assumptions, any eigenfunction $\Psi$ of \eqref{bvpsi} with eigenvalue $\lambda$ does not vanish on $(a,b)$. We let $(\alpha,\beta)$ be the smallest interval containing $(a,b)$ and such that $\Psi(\alpha)=\Psi(\beta)=0$. As $\Psi$ does not change sign on $(\alpha,\beta)$ we see that it is a first eigenfunction of \eqref{bvpsi} with Dirichlet boundary conditions on such interval. We denote by $\lambda_1(\alpha,\beta)$ the corresponding eigenvalue. Then
$$
\lambda=\lambda_1(\alpha,\beta),
$$
A standard argument of domain monotonicity shows that, as $(a,b)\subseteq(\alpha,\beta)$:
$$
\lambda_1(\alpha,\beta)\leq \lambda_1(a,b),
$$
where $\lambda_1(a,b)$ is the first eigenvalue of \eqref{bvpsi} on $(a,b)$ with Dirichlet boundary conditions. We conclude that, for all radial eigenvalues $\lambda$, one has:
$$
\lambda\leq \lambda_1(a,b),
$$
This is impossible because the radial eigenvalues form an unbounded sequence. The assertion follows. 
\end{proof}

\begin{rem}\label{alt_first_main}
Looking at the proof of Theorem \ref{s_f} we easily deduce an alternative argument for proving Theorem \ref{first_main}. In fact, note that the eigenvalues of \eqref{bvpsi} coincide with the radial spectrum. Therefore, if $\lambda$ is not a radial eigenvalue then $\Psi=0$ on $(0,D(\mathcal F))$.
\end{rem}

Theorem \ref{s_f} (Theorem \ref{sf}) is valid regardless of the fact that ${\rm Spec}(M,\mathcal F)={\rm Spec}_1(M)$. This result is proved in \cite{shklover} for compact irreducible symmetric spaces of the first rank, extending the results of \cite{ungar} for spherical caps in $\sphere 2$. However, in view of Theorem \ref{first_main}, we see that $S=(0,D(\mathcal F))$ when ${\rm Spec}(M,\mathcal F)\ne{\rm Spec}_1(M)$.


\section{Freak theorem on compact two-point homogeneous spaces}\label{sec:freak}

We will present here a simple proof of Theorem \ref{chs} based on the following {\it Addition Formula}, first proved in \cite{gine_75}, and for which we provide a short proof.  We give the proof for compact two-point homogeneous spaces: it is a classical fact that this family is equivalent to the family of compact rank one symmetric spaces, see \cite{helgason,wang}.

\begin{lemma} Let $M$ be a compact two-point homogeneous space and let $\lambda\in{\rm Spec}_1(M)$. Then:
\begin{enumerate}[i)]
\item For any $y\in M$ there is a unique normalized eigenfunction in $V(\lambda)$ which is radial  around $y$ and  positive at $y$. We denote it by  $\phi^{[y]}$.

\item For any orthonormal basis of $V(\lambda)$, say $(u_1,\dots,u_m)$, one has:
\begin{equation}\label{add_f}
\phi^{[y]}(x)=\sqrt{\dfrac{\abs M}{m}}\sum_{j=1}^mu_{j}(x)u_{j}(y)
\end{equation}
for all $x,y\in M$.
\end{enumerate}
\end{lemma}

\begin{proof} $i)$ Fix $y\in M$ and let $\phi$ be any normalized eigenfunction associated to $\lambda$. Pick a point $z\in M$ where $\phi(z)>0$; the radialization of $\phi$ around $z$ will give a normalized eigenfunction which is radial around $z$. Now fix an isometry $h$ sending $z$ to $y$: the function $\phi^{[y]}=\phi\circ h$ satisfies the assumptions.

$ii)$ Consider the function:
\begin{equation}\label{functiona}
A(x,y)=\sum_{j=1}^{m}u_{j}(x)u_{j}(y).
\end{equation}

We claim that $A(x,y)$ does not depend on the orthonormal basis chosen. 
\smallskip

In fact, if $\pi\delta_x$ denotes the orthogonal projection of the Dirac delta at $x$ onto the eigenspace $V(\lambda)$,  it is readily seen that 
$
\pi\delta_x=\sum_{j=1}^{m_k}u_j(x)u_j
$
hence
$$
A(x,y)=\scal{\pi\delta_x}{\pi\delta_y}.
$$
Consider the eigenfunction $\phi^{[y]}(x)$ as in $i)$, unique normalized eigenfunction 
which is radial around $y$ and positive at $y$. Extend it to an orthonormal basis of $V(\lambda)$  by adding $w_{2}, \dots, w_{m}$. We claim that every $w_{j}$ must vanish at $y$. In fact, if not, we can radialize $w_{j}$ around $y$ and get a radial eigenfunction orthogonal to $\phi^{[y]}$, which is impossible. Therefore, using the orthonormal basis $(\phi^{[y]}, w_2,\dots,w_m)$ to compute $A(x,y)$, we see that
\begin{equation}\label{anotherway}
A(x,y)=\phi^{[y]}(x)\phi^{[y]}(y).
\end{equation}
Now $A(y,y)=\phi^{[y]}(y)^2$ is independent on $y$ because isometries preserve eigenfunctions and their mean values over spheres. Integrating 
$A(y,y)=\sum_{j=1}^mu_j(y)^2$ over $M$ 
 we obtain $A(y,y)=\dfrac{m}{M}$ and then
\begin{equation}\label{atpoint}
\phi^{[y]}(y)=\sqrt{\dfrac{m}{M}}
\end{equation}
Putting together \eqref{functiona}, \eqref{anotherway} and \eqref{atpoint}  we obtain
$$
\phi^{[y]}(x)=\sqrt{\dfrac{\abs M}{m}}\sum_{j=1}^mu_{j}(x)u_{j}(y).
$$
\end{proof}

We remark that on $\mathbb S^1$ (i.e., in one dimension) formula \eqref{add_f} is just the addition formula for the cosine:
$$
\cos(n(x-y))=\cos(nx)\cos(ny)+\sin(nx)\sin(ny),
$$
with $n\in\mathbb N$. In this case, in \eqref{add_f} we have $\lambda=n^2$, $\phi^{[y]}(x)=\frac{\cos(n(x-y))}{\sqrt{\pi}}$, $M=2\pi$, $m=2$ (if $n\geq 1$), $u_1(x)=\frac{\cos(nx)}{\sqrt{\pi}}$, $u_2(x)=\frac{\sin(nx)}{\sqrt{\pi}}$.

\medskip
Here is what we want to prove. 
\begin{theorem} Let $M$ be a compact two-point homogeneous space with diameter $D$.
There is a countable dense set $S\subset (0,D)$ such that the geodesic ball of radius $t$ fails the Pompeiu property  if $t\in S$ and has the Pompeiu property if  $t\in (0,D)\setminus S$.
\end{theorem}

\begin{proof}
As a matter of fact, we prove the theorem for $S=S(\mathcal F)$ as in the previous section, when $\mathcal F$ is the foliation by geodesic spheres centered at a fixed point $y_0$. Thus, $\Omega_t$ is simply the ball of center $y_0$ and radius $t$, namely $\Omega_t=B(y_0,t)$.

\smallskip

If $t\in S(\mathcal F)$ then $\Omega_t$ fails the Pompeiu property by Theorem \ref{s_f}. Therefore, we only need to show that, if  $t\notin S(\mathcal F)$, then $\Omega_t$ has the Pompeiu property. 

In other words, if $f$ is continuous and satisfies $\int_{h(\Omega_t)}f=0$ for all isometries $h$, then necessarily $f=0$.

\smallskip

We fix once and for all a spectral resolution $\{u_{kj}\}$ of $M$, where $k=1,2,\dots$ and $j=1,\dots, m_k$, the multiplicity of $\lambda_k$. Expand $f$ and get:
$$
f(x)=\sum_{k=1}^{\infty}\sum_{j=1}^{m_k}a_{kj}u_{kj}(x)
$$
We want to show that, under the given assumptions, $a_{kj}=0$ for all $k,j$.
Fix the isometry $h$, and consider $h(\Omega_t)$: this is the ball of center $h(y_0)$ and same radius $t$:
$$
h(\Omega_t)=B(h(y_0),t).
$$
Let $
\chi_{h(\Omega_t)}$ denote the characteristic function of $h(\Omega_t)$. We can express it as follows:
\begin{equation}\label{chi}
\chi_{h(\Omega_t)}=\sum_{k=1}^{\infty}c_k \phi^{[h(y_0)]}_k,
\end{equation}
where $\phi^{[h(y_0)]}_k$ is the normalized  eigenfunction associated with $\lambda_k(M,\mathcal F)$, which is radial around $h(y_0)$ and positive at $h(y_0)$. In fact, just extend  $\phi^{[h(y_0)]}_k$ to an orthonormal basis of the eigenspace $V(\lambda_k)$ by adding the eigenfunctions $w_2,\dots, w_{m_k}$; any of these is orthogonal to the subspace of radial functions around $h(y_0)$, in particular, it is orthogonal to $\chi_{h(\Omega_t)}$.
Hence, only the radial eigenfunctions $\phi^{[h(y_0)]}_k$ appear in the Fourier expansion of $\chi_{h(\Omega_t)}$. 

\smallskip

Next, observe that
$$
c_k=\int_{h(\Omega_t)}\phi_k^{[h(y_0)]}=\int_{\Omega_t}\phi_k^{[y_0]}\ne 0
$$ 
for all $k$, because otherwise $t\in S(\mathcal F)$.  Note that $c_k$ does not depend on $h$. We want to express $\chi_{h(\Omega_t)}$ in terms of the fixed basis $\{u_{kj}\}$; for that, we apply the Addition Formula \eqref{add_f} to $y=h(y_0)$ and $\lambda=\lambda_k$ and obtain, for each $k$:
$$
\phi_k^{[h(y_0)]}(x)=\sqrt{\dfrac{\abs M}{m_k}}\sum_{j=1}^{m_k}u_{kj}(x)u_{kj}(h(y_0))
$$
so that \eqref{chi} becomes:
$$
\chi_{h(\Omega_t)}(x)=\sum_{k=1}^{\infty}\sum_{j=1}^{m_k}b_ku_{kj}(h(y_0))u_{kj}(x)
$$
with $b_k=\sqrt{\frac{|M|}{m_k}}c_k\ne 0$ for all $k$. We have then
$$
\begin{aligned}
\int_{h(\Omega_t)}f&=\int_M\chi_{h(\Omega_t)}f\\
&=\int_M\left(\sum_{k=1}^{\infty}\sum_{j=1}^{m_k}b_ku_{kj}(h(y_0))u_{kj}(x)\right)\left(\sum_{i=1}^{\infty}\sum_{l=1}^{m_i}a_{il}u_{il}(x)\right)\, dx\\
&=\sum_{k=1}^{\infty}\sum_{j=1}^{m_k}b_k a_{kj}u_{kj}(h(y_0)).
\end{aligned}
$$
By assumption, $\int_{h(\Omega_t)}f=0$ for any isometry $h$, hence  the right-hand side vanishes for all isometries $h$; since $G$ acts transitively on $M$, we conclude that
$$
\sum_{k=1}^{\infty}\sum_{j=1}^{m_k}b_k a_{kj}u_{kj}(y)=0
$$
for all $y\in M$.
 Multiplying by $u_{\ell m}(y)$ and integrating over $M$ we obtain that $b_{\ell}a_{\ell m}=0$, hence $a_{\ell m}=0$,  for all $\ell,m$. The proof is complete.

\end{proof}


\section{Isoparametric foliations and the Pompeiu property on the sphere}\label{sec:sphere}

When $M=\sphere n$, $n\geq 2$, is the standard $n$-sphere, we have a more explicit characterization of proper isoparametric foliations. The geometric classification of connected isoparametric hypersurfaces of $\sphere n$ is a long standing problem started with Cartan \cite{cartan1} which has been completed only very recently \cite{chi}. We resume here a few fundamental results proved in \cite{munzner1,munzner2}.

\medskip

Let $\Sigma$ be a  connected isoparametric hypersurface of $\sphere n$. Then
\begin{enumerate}[i)]
\item the number $g$ of distinct principal curvatures of $\Sigma$ can only be $1,2,3,4,6$;
\item if $\kappa_1<\cdots<\kappa_i<\cdots\kappa_g$ denote the distinct principal curvatures, their multiplicities assume only two values $m_0,m_1$ and are repeated as  $m_0,m_1,m_0,...$; moreover $n-1=\frac{g(m_0+m_1)}{2}$;
\item if $\Sigma$ has $g$ distinct principal curvatures then it is a level surface of a Cartan polynomial $p(x)$, $x\in\real {n+1}$, namely a polynomial satisfying
\begin{equation}\label{cartanp}
\begin{cases}
\overline\Delta p(x)=-c|x|^{g-2}\,,\\
|\overline\nabla p|^2=g^2|x|^{2g-2},\\
\end{cases}
\end{equation}
where $\overline\Delta$ and $\overline\nabla$ are the Laplacian and gradient of $\mathbb R^{n+1}$, and $c=(m_1-m_0)g^2/2$. Denoting $F=p_{|_{\mathbb S^n}}$, then $F(\sphere n)=[-1,1]$, $F$ is a proper isoparametric function and $\Sigma=F^{-1}(t)$ for some $t\in(-1,1)$.
\end{enumerate}
Conversely, let $p(x)$ be a Cartan polynomial satisfying \eqref{cartanp} with $g(n-1)\pm c\ne 0$, and let $F=p_{|_{\mathbb S^n}}$. Then $F^{-1}(t)$ for $t\in(-1,1)$ is a regular connected isoparametric hypersurface with $g$ distinct principal curvatures. Moreover $F$ satisfies
\begin{equation}\label{eq1}
\begin{cases}
\Delta F=g(g+n-1)F-c,\\
|\nabla F|^2=g^2(1-F^2),
\end{cases}
\end{equation}
where $\Delta$ and $\nabla$ are the usual Laplacian and gradient on $\mathbb S^n$. 

\medskip

Throughout this section, by $F$ we denote the restriction of a Cartan polynomial on $\sphere n$. As usual, let $M^+$ and $M^-$ denote the focal sets of the isoparametric foliation $\mathcal F$ associated with $F$. We can write then
\begin{equation}\label{V}
F(x)=\cos(g\rho(x)),
\end{equation}
where $\rho(x)={\rm dist}(x,M^+)$. From \eqref{V} we deduce that
$$
\Delta \rho(x)=-\frac{F(x)|\nabla F(x)|^2}{g(1-F(x)^2)^{3/2}}-\frac{\Delta F(x)}{g (1-F(x)^2)^{1/2}},
$$
which combined with \eqref{eq1} gives
\begin{equation}\label{Laplacian_sphere}
\Delta \rho(x)=-(n-1)\cot(g\rho(x))+\frac{c}{g\sin(g\rho(x))}.
\end{equation}
If $f=\psi\circ\rho$ for some smooth $\psi$, then from \eqref{ODE_f} and \eqref{Laplacian_sphere} we obtain that $\Delta f=\lambda f$ on $\mathbb S^n$ is equivalent to 
\begin{equation}\label{eq0}
\psi''(t)+\left((n-1)\cot(gt)-\frac{c}{g\sin(gt)}\right)\psi'(t)+\lambda \psi(t)=0\,,\ \ \ t\in\left(0,\frac{\pi}{g}\right),
\end{equation}
and $\psi'(0)=\psi'(\pi/g)=0$. 

\medskip

It is well-known that ${\rm Spec}_1(\sphere n)=\{k(k+n-1):k\in\mathbb N\}$. In order to apply Theorem \ref{first_main} we need to compute ${\rm Spec}(\sphere n,\mathcal F)$.

\begin{thm}\label{sphere_prop}
Let $(\sphere n,\mathcal F)$ be an isoparametric foliation with $g$ distinct principal curvatures. Then 
$$
{\rm Spec}(\sphere n,\mathcal F)=\{gk(gk+n-1):k\in\mathbb N\}.
$$
\end{thm}

\begin{proof}
By setting $\psi(t)=y(\cos(gt))$, problem \eqref{eq0} is recast to the following singular Sturm-Liouville problem
\begin{equation}\label{SSL2}
(1-x^2)y''(x)+\left(\frac{c}{g^2}-\left(\frac{n-1}{g}+1\right)x\right)y'(x)+\frac{\lambda}{g^2}y(x)=0,
\end{equation}
for $x\in(-1,1)$. We require that $y$, along with its derivatives, remains bounded at $x=\pm 1$.

\medskip

{\bf Part 1.} First, we show that the eigenvalues are simple. In fact, if $y_1,y_2$ are two bounded solutions of \eqref{SSL2} with same eigenvalue $\lambda$ then
$$
(1-x)^{\frac{n-1}{2g}-\frac{c}{2g^2}+\frac{1}{2}}(1+x)^{\frac{n-1}{2g}+\frac{c}{2g^2}+\frac{1}{2}}(y_1'(x)y_2(x)-y_2'(x)y_1(x))=C
$$
on $(-1,1)$, for some $C\in\mathbb R$. By letting $x\rightarrow\pm 1$ we deduce that $C=0$, since $\frac{n-1}{g}\pm\frac{c}{g^2}\geq 1$. Then $y_1$ and $y_2$ are linearly dependent.

\medskip

{\bf Part 2.}  We prove now that \eqref{SSL2} admits a polynomial solution if and only if $\lambda=gk(gk+n-1)$ for $k\in\mathbb N$. To simplify the presentation, we will consider the case $c=0$. The case $c\ne 0$ is treated similarly.  Let $y$ be any function satisfying \eqref{SSL2} for some $\lambda\in\mathbb R$. It is analytic in $(-1,1)$ hence we can write $y(x)=\sum_{j=0}^{\infty}a_j x^j$. Substituting this expression in \eqref{SSL2} we obtain
\begin{equation}\label{exp}
(1-x^2)\sum_{j=2}^{\infty}j(j-1)a_jx^{j-2}-\sum_{j=1}^{\infty}\left(\left(\frac{n-1}{g}+1\right)x\right)ja_jx^{j-1}+\sum_{j=0}^{\infty}\frac{\lambda}{g^2} a_jx^j=0.
\end{equation}
Comparing the coefficients we find the recurrence relations 
\begin {equation}\label{recurr}
a_2=-\frac{\lambda}{2g^2}a_0\,,\ \ \  a_3=\frac{1}{6}\left(\left(\frac{n-1}{g}+1\right)-\frac{\lambda}{g^2}\right)a_1\,,\ \ \ 
a_{j+2}=\frac{j\left(j+\frac{n-1}{g}\right)-\frac{\lambda}{g^2}}{(j+2)(j+1)}a_j.
\end{equation}
If $y$ is a polynomial of degree $k$, then from \eqref{recurr} with $j=k$ we see that $\lambda=gk(gk+n-1)$. Conversely, if $\lambda=gk(gk+n-1)$ for some $k\in\mathbb N$, then $a_{k+2\ell}=0$ for all $\ell\geq 1$. Without loss of generality, assume $k$ even. Then $y(x)=P_k(x)+\sum_{j=0}^{\infty}a_{2j+1}x^{2j+1}$, where $P_k$ is an even polynomial of degree $k$. If $a_{2j+1}=0$ for some $j$, then from \eqref{recurr} we see that $a_{2j+1}=0$ for all $j\in\mathbb N$. Assume that $a_{2j+1}\ne 0$ for some $j\in\mathbb N$ (hence for all $j\in\mathbb N$). Relation \eqref{recurr} is rewritten as
\begin{equation}\label{recurr2} 
a_{2j+3}=\frac{(2j+1)\left(2j+1+\frac{n-1}{g}\right)-\frac{\lambda}{g^2}}{(2j+3)(2j+2)}a_{2j+1}.
\end{equation}
We note that there exists $j_0\in\mathbb N$ such that $(2j+1)(2j+1+\frac{n-1}{g})-\frac{\lambda}{g^2}>0$ and $a_{2j+1}$ does not change sign for all $j\geq j_0$. Without loss of generality, assume $a_{2j+1}>0$ for all $j\geq j_0$. Since $\frac{n-1}{g}\geq 1$, we have, for all $j\geq j_0$
\begin{multline*}
a_{2j+1}\geq\frac{(2j+1)(2j+2)-\frac{\lambda}{g^2}}{(2j+3)(2j+2)}\cdots \frac{(2j_0+1)(2j_0+2)-\frac{\lambda}{g^2}}{(2j_0+3)(2j_0+2)}a_{2j_0+1}\\=\frac{2j_0+1}{2j+3}\prod_{\ell=j_0}^j\left(1-\frac {\lambda}{g^2(2\ell+1)(2\ell+2)}\right)
\end{multline*}
Since $\prod_{\ell=0}^{\infty}\left(1-\frac {\lambda}{g^2(2\ell+1)(2\ell+2)}\right)$ converges to a positive number, we deduce that there exists  $C>0$ such that $a_{2j+1}\geq\frac{C}{2j+3}$ for all $j\geq j_0$. Then $\lim_{x\rightarrow\pm 1}y(x)=\pm \infty$. Thus $a_{2j+1}$ must be zero for all $j\in\mathbb N$ and $y=P_k$.

\medskip

{\bf Part 3.} Assume now $\lambda\ne gk(gk+n-1)$ for all $k\in\mathbb N$ and let $y\ne 0$ be a bounded eigenfunction associated to $\lambda$. Then 
\begin{equation}\label{SW}
\int_{-1}^1(1-x)^{\frac{n-1}{2g}-\frac{c}{2g^2}-\frac{1}{2}}(1+x)^{\frac{n-1}{2g}+\frac{c}{2g^2}-\frac{1}{2}}y(x)P_k(x)dx=0
\end{equation}
for all $k\in\mathbb N$, where $P_k(x)$ is an eigenfunction associated to the eigenvalue $gk(gk+n-1)$, i.e., a polynomial of degree $k$. By the Stone-Weierstrass Theorem we know that $y$ can be uniformly approximated in $C^0([-1,1])$ by polynomials. Therefore we can replace $P_k$ by $y$ in \eqref{SW} and deduce that
$$
\int_{-1}^1(1-x)^{\frac{n-1}{2g}-\frac{c}{2g^2}-\frac{1}{2}}(1+x)^{\frac{n-1}{2g}+\frac{c}{2g^2}-\frac{1}{2}}y(x)^2dx=0
$$
hence $y=0$.

\end{proof}

\begin{rem}
Standard computations show that the polynomials defined by
\begin{equation}\label{polynomial}
P_k(x)=\frac{1}{(1-x)^{\frac{n-1}{2g}-\frac{c}{2g^2}-\frac{1}{2}}(1+x)^{\frac{n-1}{2g}+\frac{c}{2g^2}-\frac{1}{2}}}\frac{d^k}{d x^k}\left((1-x)^{\frac{n-1}{2g}-\frac{c}{2g^2}-\frac{1}{2}+k}(1+x)^{\frac{n-1}{2g}+\frac{c}{2g^2}-\frac{1}{2}+k}\right)
\end{equation}
are bounded eigenfunctions of \eqref{SSL2} associated with $\lambda_k=gk(gk+n-1)$. For $k=0,1,2,3$, we have
\begin{eqnarray*}
P_0(x)&=&1,\\
P_1(x)&=&\frac{c}{g^2}-\left(1+\frac{n-1}{g}\right)x\\
P_2(x)&=&\frac{c^2}{g^4}-\left(3+\frac{n-1}{g}\right)-\frac{2c}{g^2}\left(2+\frac{n-1}{g}\right)x+\left(2+\frac{n-1}{g}\right)\left(3+\frac{n-1}{g}\right)x^2\\
P_3(x)&=&\frac{c^3}{g^6}-\frac{c}{g^2}\left(13+\frac{3(n-1)}{g}\right)\\
&&+\left(3\left(3+\frac{n-1}{g}\right)\left(5+\frac{n-1}{g}\right)-\frac{3c^2}{g^4}\left(3+\frac{n-1}{g}\right)\right)x\\
&&+\frac{3c}{g^2}\left(3+\frac{n-1}{g}\right)\left(4+\frac{n-1}{g}\right)x^2\\
&&-\left(3+\frac{n-1}{g}\right)\left(4+\frac{n-1}{g}\right)\left(5+\frac{n-1}{g}\right)x^3.
\end{eqnarray*}
Note that for $c=0$, $P_k(x)$ is an even polynomial for $k$ even and it is odd for $k$ odd. 
\end{rem}

\begin{cor}\label{cor_sphere}
Let $\mathcal F$ be an isoparametric foliation of $\mathbb S^n$ with $g\ne 1$. Then every isoparametric tube of $\mathcal F$ fails the Pompeiu property.
\end{cor}
\begin{proof}
The corollary follows from Theorem \ref{first_main}. If $g\ne 1$, then, from Theorem \ref{sphere_prop} we have ${\rm Spec}(\mathbb S^n,\mathcal F)=\left\{gk(gk+n-1):k\in\mathbb N\right\}\subsetneq{\rm Spec}_1(\mathbb S^n)=\left\{k(k+n-1):k\in\mathbb N\right\}$.
\end{proof}

When $g=1$ the isoparametric foliations are given by concentric spheres, and all the eigenvalues of the Laplacian are radial. In particular, up to isometries, $F={x_1}_{|_{\sphere n}}$. Any eigenspace contains exactly one radial eigenfunction (up to scalar multiplication).

\medskip

When $g=2$, $n\geq 3$ isoparametric foliations are given by Clifford tori. Namely, up to isometries, $F=p_{|_{\sphere n}}$, where $p(x)=\sum_{i=1}^{\ell}x_i^2-\sum_{i=\ell}^{\ell+m}x_i^2$, $n+1=\ell+m$, $\ell,m>1$. The radial eigenvalues are $\{2k(2k+n-1):k\in\mathbb N\}$.

\medskip

As already mentioned, Theorems \ref{first_main} and \ref{s_f} hold also in the case of non proper isoparametric foliations. In this case, the focal sets can have codimension $1$, or they can be disconnected. For example, on $\sphere n\subset\real {n+1}$ we consider the function $F={(1-x_1^2)^{\frac{k}{2}}}_{|_{\sphere n}}$, where $x=(x_1,..,x_{n+1})$ denotes a point in $\real {n+1}$, $k\in\mathbb N$. For $k\geq 2$ the function $F$ is isoparametric but it is not proper. For simplicity, let us consider $k=2$. If $x_0=(1,0,\cdots,0)$, then $F=\sin(\rho(x))^2$, where $\rho(x)={\rm dist}(x,x_0)$. The focal sets are $M^+=\{x_0\}\cup\{-x_0\}$ and $M^-=\{\rho(x)=\frac{\pi}{2}\}$ (i.e., the equator). We see that one focal set is disconnected, while the other has codimension $1$.

\medskip

Isoparametric tubes $\Omega_t$ with boundary $F^{-1}(t)$, $t\in\left(0,\frac{\pi}{2}\right)$, are either equidistant bands around the equator, or opposite congruent spherical caps. Theorems \ref{first_main} and \ref{s_f} hold with no modifications. In particular, one can verify that ${\rm Spec}(\sphere n,\mathcal F)\subsetneq{\rm Spec}_1(\sphere n)$, hence $\Omega_t$ fails the Pompeiu property for all $t\in\left(0,\frac{\pi}{2}\right)$. In fact the coordinate functions ${x_i}_{|_{\sphere n}}$ are not in the radial spectrum and integrate zero over all $\Omega_t$. In particular, the origin of $\real {n+1}$ is always the barycenter of $\Omega_t$ (see also Section \ref{sec:even}).

\section{A simple argument using antipodal invariance}\label{sec:even}

In this section we show, by very simple arguments,  that the Pompeiu property fails for a rather large class of spherical domains, namely, those domains which are invariant under the antipodal map. In what follows, we call a function $f:\mathbb R^{n+1}\rightarrow\mathbb R$ {\it linear} if it is of the form $f(x_1,...,x_{n+1})=\sum_ja_jx_j$ for some $a_1,...,a_{n+1}\in\mathbb R$.

\begin{prop}\label{peven} Let $\Omega$ be an antipodal invariant domain in $\sphere n$. Then:

\begin{enumerate}[i)]

\item One has $\int_{\Omega}u=0$ for all linear functions $u$. In particular, the barycenter of $\Omega$ is always the origin.

\item  $\Omega$ fails the Pompeiu property.
\end{enumerate}
\end{prop} 

\begin{proof} 
$i)$ For any isometry $\sigma$ of the sphere and any function $u$ one has
$
\int_{\sigma(\Omega)}u=\int_{\Omega}u\circ \sigma.
$
Let $\sigma(x)=-x$ be the canonical involution; it is an isometry and by assumption $\sigma(\Omega)=\Omega$. Now, for any linear function $u$, we have $u\circ\sigma=-u$; this implies that:
$$
\int_{\Omega}u=\int_{\sigma(\Omega)}u
=\int_{\Omega}u\circ\sigma
=-\int_{\Omega}u
$$
which shows the assertion. 

\smallskip

$ii)$ We know that the space of linear functions is invariant by the isometry group of $\sphere n$ (it corresponds to the eigenspace associated to $\lambda_2(\sphere n)$). Fix a linear function $u$ and let $h$ be any isometry. Then:
$$
\int_{h(\Omega)}u=\int_{\Omega}u\circ h=0
$$
because $u\circ h$ is linear as well. 
\end{proof}

A simple consequence is that, if $F:\real{n+1}\to \reals$ is any polynomial function of even degree, and if
$
\Omega=\{x\in\sphere n: F(x)<c\}
$
for some $c$, then $\Omega$ is antipodal invariant simply because $F(x)=F(-x)$ for all $x\in\sphere n$. Since any isoparametric tube with $g$ even is of that type ( its Cartan polynomial has in fact degree $g$), we immediately get

\begin{cor} Any isoparametric tube with $g=2,4,6$ fails the Pompeiu property. 
\end{cor}

It remains to deal with the cases $g=1$ and $g=3$. When $g=1$ we are dealing with the standard isoparametric foliation, and isoparametric tubes are geodesic balls:  by the Freak Theorem (or Theorem \ref{chs}) we know that in that case $\Omega=B(r)$ fails the Pompeiu property only for $r$ in a countable and dense subset of $(0,\pi)$. In  the case $g=3$ (in which case isoparametric tubes are never antipodal invariant),  the fact that all isoparametric tubes fail the Pompeiu property is a consequence of Theorem \ref{first_main}, as explained in Corollary \ref{cor_sphere}.

\section{An upper bound for the first positive eigenvalue of isoparametric hypersurfaces}\label{yau}

A well known conjecture of Yau \cite{yau} states that the first positive eigenvalue $\lambda_2(\Sigma)$ of any embedded minimal hypersurface $\Sigma$ of $\sphere n$ is $n-1$. In \cite{tang_yan_yauconj_2,tang_yan_yauconj_1} Yau's conjecture is proved for minimal isoparametric hypersurfaces. Theorem \ref{first_main} and Theorem \ref{sphere_prop} allow to prove that for $g\ne 1$ the minimal regular hypersurface is the maximizer of $\lambda_2(\Sigma)$ in the isoparametric family.

When $g=1$ isoparametric hypersurfaces are geodesic spheres and $n-1$ is actually a minimizer for $\lambda_2(\Sigma)$ which takes values in $[n-1,+\infty)$.

\begin{theorem}\label{yauth}
Let $\Sigma$ be a connected isoparametric hypersurface of $\sphere n$ with $g>1$ distinct principal curvatures. Then 
$$
\lambda_2(\Sigma)\leq n-1.
$$
Equality holds if and only if $\Sigma$ is minimal.
\end{theorem}
\begin{proof}
The restrictions of the coordinate functions $x_1,...,x_{n+1}$ to $\sphere n$ are eigenfunctions of the Laplacian on $\sphere n$ with eigenvalue $n$. From Proposition \ref{sphere_prop} it follows that if $g>1$ then $n\notin{\rm Spec}(\sphere n,\mathcal F)$. From Theorem \ref{first_main} we deduce that
$$
\int_{\Sigma}x_i=0\,,\ \ \ i=1,...,n+1.
$$
Then
\begin{equation}\label{yau1}
\lambda_2(\Sigma)\leq\frac{\int_{\Sigma}|\nabla_{\Sigma}x_i|^2}{\int_{\Sigma}x_i^2}\,,\ \ \ i=1,...,n+1.
\end{equation}
Here $\nabla_{\Sigma}$ denotes the gradient on $\Sigma$. Multiplying both sides of \eqref{yau1} by $\int_{\Sigma}x_i^2$ and summing over $i=1,...,n+1$ we obtain
$$
|\Sigma|\lambda_2(\Sigma)\leq\int_{\Sigma}\sum_{i=1}^{n+1}|\nabla_{\Sigma} x_i|^2.
$$
The result follows since
$$
\sum_{i=1}^{n+1}|\nabla_{\Sigma}x_i|^2=\sum_{i=1}^{n+1}\left(|\overline\nabla x_i|^2-|\langle\overline\nabla x_i, x\rangle|^2-|\langle\overline\nabla x_i,\nabla\rho(x)\rangle|^2\right)=n-1.
$$
Here $\rho(x)={\rm dist}(x,M^+)$ with $M^+$ a focal set of $\Sigma$, ${\rm dist}$ and $\nabla$ are the usual distance and gradient on $\sphere n$ and $\overline\nabla$ is the standard gradient of $\mathbb R^{n+1}$.

\medskip

If $\Sigma_{min}$ is minimal then from \cite{tang_yan_yauconj_2} we have $\lambda_2(\Sigma_{min})=n-1$.  Conversely, if $\lambda_2(\Sigma)=n-1$, then $x_1,...,x_{n+1}$ is a set of corresponding eigenfunctions. A classical result by Takahashi  \cite{takahashi} implies that $\Sigma_{min}$ is minimal.
\end{proof}

\section{One-dimensional case and the role of compactness}\label{A}

We have seen from \cite{berenstein_zalcman_1,ungar} (alternatively, from Theorems \ref{sf} and \ref{chs}) that there exists a countable and dense set of radii $S\subsetneq(0,\pi)$ such that the spherical cap  $B(r)\subset\mathbb S^n$ (the ball of radius $r$) fails the Pompeiu property whenever $r\in S$, while it has the Pompeiu property whenever $r\notin S$. On the other hand, on $\mathbb R^n$ all balls fail the Pompeiu property. As already observed, this difference is quite striking and is certainly one of the reasons why Ungar, the first to have noticed it for $\mathbb S^2$, called this result the ``Freak Theorem''.

\medskip

This difference seems to be intimately connected with the compactness of the space, and consequently with the discreteness of the spectrum of the Laplacian. We want to provide here a very simple interpretation of this fact, based on elementary computations in one dimension.

\medskip

First, we note that all balls in $\mathbb R$ (i.e., all segments) fail the Pompeiu property, exactly as it happens in $\mathbb R^n$. In fact, for all values of $\alpha>0$ we find non-zero functions $f$ such that
$$
\int_c^{c+2\alpha}f(x)dx=0
$$
for all $c\in\mathbb R$. It is sufficient to take any $2\alpha$-periodic function on $\mathbb R$ which integrates to zero on $[0,2\alpha]$, for example
$f(x)=a \cos(\frac{\pi}{\alpha}x)+b\sin(\frac{\pi}{\alpha}x)$. In this case we note that the spectrum of the Laplacian on $\mathbb R$ is the whole interval $[0,+\infty)$, and $f$ is a generalized eigenfunction with eigenvalue $\frac{\pi^2}{\alpha^2}$.

\medskip

On the other hand, Ungar's Freak Theorem can be appreciated already in the case of $\mathbb S^1$, and the corresponding proof is an easy exercise. 

\begin{prop}
Let $\alpha\in(0,\pi)$ and let $f$ be a continuous functions on $\mathbb S^1$ such that $\int_c^{c+2\alpha}f(x)dx=0$ for all $c\in(-\pi,\pi)$. Then $f=0$  unless $\frac{\alpha}{\pi}\in (0,1)\cap\mathbb Q$. 
\end{prop}
\begin{proof}
We write
\begin{equation}\label{exp1}
f(x)=a_0+\sum_{j=1}^{\infty}(a_j\cos(jx)+b_j\sin(jx)),
\end{equation}
with $a_0^2+\sum_{j=1}^{\infty}(a_j^2+b_j^2)<+\infty$. Then
\begin{equation}\label{cond1}
\int_c^{c+2\alpha}f(x)dx=2a_0\alpha+\sum_{j=1}^{\infty}\frac{2}{j}\sin\left(j\alpha\right)\left(a_j(\cos(j(c+\alpha)))+b_j(\sin(j(c+\alpha)))\right).
\end{equation}

If $\frac{\alpha}{\pi}=\frac{m}{n}\in (0,1)\cap\mathbb Q$ ($m,n\in\mathbb N$, $n\ne 0$), we choose $f(x)=a_n\cos(nx)+b_n\sin(nx)$, $(a_n,b_n)\ne (0,0)$. The sum \eqref{cond1} is then zero for all $c\in(-\pi,\pi)$, i.e., arcs of length $2\alpha$ fail the Pompeiu property.

\medskip

Let now $\alpha=\pi s$ with $s\in(0,1)\setminus\mathbb Q$. We set
\begin{equation}\label{ga0}
g(c):=\int_c^{c+2\alpha}f(x)dx=2\pi s a_0+\sum_{j=1}^{\infty}\frac{2}{j}\sin(j\pi s)\left(a_j(\cos(jc+j\pi s)))+b_j(\sin(jc+j\pi s))\right).
\end{equation}
 Then $g(c)=A_0+\sum_{j=1}^{\infty}A_j\cos(jc)+B_j\sin(jc)$, where the coefficients are given by $A_0=a_0\alpha$, $A_j=\frac{2\sin(\pi s)}{j}\left(a_j\cos(j\pi s)+b_j\sin(j\pi s)\right)$, $B_j=\frac{2\sin(\pi s)}{j}\left(b_j\cos(j\pi s)-a_j\sin(j\pi s)\right)$. Assume $g(c)=0$ for all $c\in(-\pi,\pi)$. The series defining $g$ is absolutely convergent hence $A_0=0$ and $A_j=B_j=0$ for all $j\in\mathbb N$. Since $\sin(j\pi s)\ne 0$, this happens if and only if $a_0=0$ and $a_j=b_j=0$ for all $j\geq 1$. Therefore $f=0$.
\end{proof}
Note that in the case of $\mathbb S^1$ and the foliation with focal set given by a point, the set $S(\mathcal F)$ of Theorem \ref{sf} is $\{\pi s:s\in(0,1)\cap\mathbb Q\}$. Then balls (arcs) fail the Pompeiu problem only for a countable and dense subset. The corresponding functions integrating zero on all arcs of length $2\alpha=\frac{2\pi m}{n}$ ($m,n\in\mathbb N$, $n\ne 0$) are the eigenfunctions of the Laplacian on $\mathbb S^1$ corresponding to the eigenvalues $\ell^2$, for $\ell\in\mathbb N$, $\ell\geq n$.

\medskip

We conclude this section with a brief discussion on other (non proper) isoparametric foliations on $\sphere 1$ for which the corresponding tubes are not connected and yet the results of Theorem \ref{firstmain} hold as well. These examples may be not interesting per se, however they are useful to have a simple interpretation of Theorem \ref{firstmain} and Corollary \ref{corsphere}.

\medskip

Let $x\in\sphere 1$ and let $F(x)=\cos(kx)$. The function $F$ is isoparametric and the corresponding tubes of radius $t\in(0,\pi/k)$ with focal points at $F=\pm 1$ are given by $\Omega_t(0)$, where
$$
\Omega_t(c):=\bigcup_{j=0}^{k-1}\left(c+\frac{\pi}{2k}+\frac{j\pi}{k}-t,c+\frac{\pi}{2k}+\frac{j\pi}{k}+t\right),
$$
Choosing as test functions $f(x)= a \cos(\ell x)+b\sin(\ell x)$ we see that $\int_{\Omega_t(c)}f(x)dx=0$ for all $c\in(0,\pi/k)$ if and only if $\frac{\ell^2}{k^2}\ne m^2$, with $m\in\mathbb N$, and we can always find such an $\ell$ when $k\geq 2$. We remark that this condition is exactly the analogous of the condition described in Corollary \ref{cor_sphere}, namely $\{gm(gm+n-1):m\in\mathbb N\}\subsetneq\{\ell(\ell+n-1):\ell\in\mathbb N\}$ with $n=1$ and $g=k$.



\section*{Acknowledgements}
The authors are members of the Gruppo Nazionale per le Strutture Algebriche, Geometriche e le loro Applicazioni (GNSAGA) of the I\-sti\-tuto Naziona\-le di Alta Matematica (INdAM). 

\bibliography{bibliography}
\bibliographystyle{abbrv}

\vspace{1.0cm}

Luigi Provenzano and Alessandro Savo, \\
Dipartimento di Scienze di Base e Applicate per l'Ingegneria, \\
Sapienza Universit\`{a} di Roma, \\
Via Antonio Scarpa,  16, 
00161 Roma, 
Italy.
\smallskip

e-mail:	alessandro.savo@uniroma1.it \\
e-mail: luigi.provenzano@uniroma1.it


\end{document}